\documentclass[11pt]{amsart}
\usepackage{caption}
\usepackage{subcaption}
\usepackage{color}
\usepackage{pdfpages}
\usepackage{tikz}
\usepackage{fullpage}
\usepackage[hidelinks]{hyperref}
\hypersetup{
  colorlinks   = true, 
  urlcolor     = orange, 
  linkcolor    = brown, 
  citecolor   = blue 
}

\usepackage{amsmath}
\usepackage{amssymb}
\usepackage{amsthm}
\usepackage{mathrsfs}
\usepackage{mathtools}
\usepackage{bigints}
\usepackage{enumitem}
\usepackage{hyperref}
\usepackage{isomath}
\usepackage[toc,page]{appendix}
\usepackage{multicol}
\usepackage{seqsplit}
\usepackage{array}
\usepackage[all]{xy}
\usepackage{tikz-cd}
\usepackage{spverbatim,comment}

\newtheorem{sublem}{Lemma}[section]
\newtheorem{thm}{Theorem}
\newtheorem{subthm}{Theorem}[section]

\newtheorem{prop}{Proposition}
\newtheorem{subprop}{Proposition}[section]
\theoremstyle{definition}
\newtheorem{remark}{Remark}

\newtheorem*{example}{Example}

\newcommand{\cD}{\lceil D\rceil}
\renewcommand{\O}{\mathcal{O}}
\newcommand{\M}{\mathcal{M}}
\newcommand{\PP}{\mathbb{P}}
\newcommand{\CC}{\mathbb{C}}
\newcommand{\QQ}{\mathbb{Q}}
\newcommand{\ZZ}{\mathbb{Z}}
\newcommand{\HH}{\mathbb{H}}
\renewcommand{\l}{\tilde{l}}
\usepackage[
backend=biber,
style=alphabetic,
sorting=nyt
]{biblatex}
\numberwithin{equation}{section}

\title{Vanishing Theorem for Hodge Ideals on Smooth Hypersurfaces}

\author{Anh Duc Vo}

\newcommand{\Addresses}{{
  \vspace{\bigskipamount}
  \footnotesize
  \textsc{Department of Mathematics, Harvard University, Cambridge, MA 02138, USA}\par\nopagebreak
  \textit{E-mail address}: \texttt{ducvo@math.harvard.edu}
}}

\begin{document}

\maketitle

\begin{abstract}
    We use a Koszul-type resolution to prove a weak version of Bott's vanishing theorem for smooth hypersurfaces in $\PP^n$ and use this result to prove a vanishing theorem for Hodge ideals associated to an effective Cartier divisor on a hypersurface. This extends an earlier result of Musta\c t\u a  and Popa.    
\end{abstract}

\tableofcontents

\section{Introduction}

\indent Let $X$ be a smooth complex variety, and $D$ an effective $\QQ$-divisor on $X$. Recently, Musta\c t\u a and Popa introduced the notion of Hodge ideals $I_k(D)$ in a series of papers \cite{Hodge_ideals}, \cite{Hodge_ideals_Q1} and \cite{Hodge_ideals_Q2}. These ideals are generalizations of multiplier ideals; in fact, one can show that $I_0(D) = \mathcal{J}((1-\epsilon)D)$, the multiplier ideal of $\mathbb{Q}$-divisor $(1-\epsilon)D$, for any $0 < \epsilon \ll 1$. They were introduced with the goal of understanding the singularities and Hodge theory of divisors on smooth varieties; for example, \cite[Corollary H]{Hodge_ideals} gives the bound for the degree of hypersurfaces in  $\PP^n$ on which a set of isolated singular points of a divisor imposes independent conditions. 

\vspace{\medskipamount}

Multiplier ideals satisfy the celebrated Nadel vanishing theorem; see \cite[Theorem 9.4.8]{positivity2}. In their paper \cite[Theorem F]{Hodge_ideals}, Musta\c t\u a and Popa obtained an analogous result for Hodge ideals, but required some technical assumptions. When $X$ is a projective space or an abelian variety, they also pointed out that these assumptions are in fact not necessary \cite[Theorem 25.3 and 28.2]{Hodge_ideals}. In this paper, we improve their vanishing result when $X$ is a smooth hypersurface in $\PP^n$. We note that a stronger result than the general statement has also been proven by Dutta in the case of a toric variety \cite[Theorem A]{vanishing_toric}. 

\vspace{\medskipamount}

To formulate the problem, let $X$ be a smooth hypersurface of degree $d\ge 2$ in $\PP^n$ with $n \ge 4$, and $D$ an effective $\mathbb{Q}$-divisor on $X$. A special property of projective spaces and toric varieties, which allows us to understand the cohomology of Hodge ideals, is that there exists an Euler-type exact sequence. This exact sequence can be used to construct a Koszul-type resolution of sheaves of differential forms $\Omega^p_X$. Hypersurfaces in $\PP^n$ inherit the Euler sequence via restriction; however, these do not give a resolution of $\Omega^p_X$, but $\Omega^p_{\PP^n}|_X$. This makes it challenging to obtain a nice set of conditions for the vanishing theorem; see Section \ref{sec_arbsing}. Fortunately, if we require the divisor $D$ to have at worst isolated singularities, i.e. the support of $D$ has at worst isolated singularities, we obtain the following theorems in Sections \ref{sec_iso4} and \ref{sec_iso3}. Notice that if $n\ge 4$, any line bundle on $X$ is the restriction of a line bundle on $\PP^n$ due to the Grothendieck-Lefschetz theorem \cite[Corollary 3.3]{amplesubvar}, so we consider the cases $n\ge 4$ and $n=3$ separately. 

\begin{thm} \label{vanishing_isolated_sing}
    Let $X$ be a smooth hypersurface of degree $d\ge 2$ in $\PP^n$, with $n \ge 4$, and $D$ an effective $\mathbb{Q}$-divisor on $X$ of degree $m$, where $m$ is a rational number. Assume further that $D$ has, at worst, isolated singularities. Denote by $Z$ the support of $D$, and by $H$ the hyperplane divisor on $X$. Let  $\O_X(Z-\cD)\cong \O_X(-a)$ for some non-negative integer $a$. Then: 
    \begin{enumerate}
        \item For $2 \le i \neq n-1$, $k\ge 0$ and all $l$, we have 
        \[
            H^i\big(X, \omega_X(kZ) \otimes I_k(D) \otimes \O_X(l)\big) = 0.
        \]
        \item For $k \ge 0$ and $(l-a)H + kZ$ ample, we have 
        \[
            H^{n-1}\big(X, \omega_X(kZ) \otimes I_k(D) \otimes \O_X(l)\big) = 0.
        \]
        \item For $k \ge 1$, $l \ge \max\{k(d-2) + m, k(d-2) + (n-2)(d-1) + a\}$, we have 
        \[
            H^{1}\big(X, \omega_X(kZ) \otimes I_k(D) \otimes \O_X(l)\big) = 0.
        \]
    \end{enumerate}
\end{thm}

\noindent Note that when $D$ is a reduced integral divisor (as in \cite{Hodge_ideals_Q1}), we have $a = 0$ and  $Z=D$, and the conditions (2) and (3) become
\begin{itemize}
    \item For $k \ge 0$ and $l + km > 0$, we have 
        \[
            H^{n-1}\big(X, \omega_X(kD) \otimes I_k(D) \otimes \O_X(l)\big) = 0.
        \]
    \item For $k \ge 1$, $l \ge \max\{k(d-2) + m, k(d-2) + (n-2)(d-1)\}$, we have 
        \[
            H^{1}\big(X, \omega_X(kD) \otimes I_k(D) \otimes \O_X(l)\big) = 0.
        \]
\end{itemize}

\begin{thm} \label{vanishing_surface}
    Let $X$ be a smooth hypersurface in $\PP^3$ of degree $d \ge 2$, $D$ an effective $\QQ$-divisor, and $L$ a line bundle on $X$. Then
    \begin{enumerate}
        \item For $i\ge 2$, if $L((k+1)Z-\cD)$ is ample, then  
        \[
            H^i\big(X, \omega_X(kZ) \otimes I_k(D) \otimes L\big) = 0.
        \]
        \item For $i=1$, if $(L-\cD)\big((k-1)(2-d)\big) \text{ and } L(jZ-\cD)\big((k-j+2)(2-d)\big)$ are ample for $1 \le j \le k$, then 
        \[
            H^1\big(X, \omega_X(kZ) \otimes I_k(D) \otimes L\big) = 0.
        \]
    \end{enumerate}
\end{thm}

\noindent Similarly, when $D$ is a reduced integral divisor, we have instead
\begin{itemize}
    \item For $i\ge 2$, if $L(kD)$ is ample, then  
        \[
            H^i\big(X, \omega_X(kD) \otimes I_k(D) \otimes L\big) = 0.
        \]
    \item For $i=1$, if $(L-D)\big((k-1)(2-d)\big) \text{ and } L(jD)((k-j+1)(2-d))$ are ample for $0 \le j \le k-1$, then 
        \[
            H^1\big(X, \omega_X(kZ) \otimes I_k(D) \otimes L\big) = 0.
        \]
\end{itemize}

\begin{remark} \label{I_0}
\begin{enumerate}
    \item Since $I_0(D) \cong \mathcal{J}((1-\epsilon)D)$, the vanishing for $I_0(D)$ is a direct consequence of the Nadel vanishing theorem which says
    \[
        H^i\big(X, \omega_X \otimes I_0(D) \otimes L) \cong H^i(X, \omega_X \otimes \mathcal{J}((1-\epsilon)D) \otimes L\big) = 0
    \] 
    if $L - (1-\epsilon) D$ is ample for any $0 < \epsilon \ll 1$ and $i\ge 1$. In particular, if $n \ge 4$, we have 
    \[
        H^i\big(X, \omega_X \otimes I_0(D) \otimes \O_X(l))\big) = 0,
    \]
    for $l \ge m$. On the other hand, if $n = 3$, we have 
    \[
        H^i\big(X, \omega_X \otimes I_0(D) \otimes L\big) = 0,
    \]
    for $L-D$ nef.  
    
    \item The case $d=1$, or equivalently $X$ is the projective space $\PP^{n-1}$, is a result due to Musta\c t\u a and Popa \cite[Variant 12.5]{Hodge_ideals_Q1}, and Chen \cite[Variant 4.5]{BingYi}. 
\end{enumerate}
\end{remark}

As an application, by utilizing the above vanishing theorems, we investigate the classical problem of finding a bound for the degree of hypersurfaces in $\PP^n$ on which a set of points imposes independent conditions. This extends a result of Musta\c t\u a and Popa, \cite[Theorem G, Corollary H]{Hodge_ideals_Q1}. 


\begin{thm} \label{bound of deg}
    Let $D$ be a complete intersection of a smooth irreducible hypersurface $X$ of degree $d \ge 2$ and another hypersurface (not necessarily smooth) of degree $m$ in $\PP^n$, with $n \ge 4$. Assume furthermore that $D$ has at worst isolated singularities.
    \begin{enumerate}
        \item For each $k$, denote by $Z_k$ the subscheme associated with the Hodge ideal $I_k(D)$, then  
        \[
            \text{length}(Z_k) \le \binom{(k+1)m +l +d - 1}{n} - \binom{(k+1)m +l - 1}{n},
        \]
        
        \noindent where  $l = \max \{k(d-2), k(d-2) + (n-2)(d-1) - m\}$. 
        \item Denote by $S_t$ the set of singular points on $D$ of multiplicity at least $t \ge 2$. Then $S_t$ imposes independent conditions on hypersurfaces of degree at least 
        \[
            (k+1)m + l + d - n - 1,
        \]
        in which $k = \lceil \frac{n}{t} \rceil -1$ and  $l = \max \{0, k(d-2), k(d-2) + (n-2)(d-1) - m\}$. 
    \end{enumerate}
\end{thm}

The vanishing for Hodge ideals in Theorem \ref{vanishing_surface} can be used to obtain an analog of Theorem $C$ for the case $n=3$. However, the singularities of curves are well-understood with the help of multiplier ideals and adjoint ideals; hence, we will consider only the case $n \ge 4$. 

\vspace{\bigskipamount}

The main ingredient in the proofs of the above results is what we call "borderline Nakano vanishing". Recall the Nakano vanishing theorem \cite[Theorem 4.2.3]{positivity1} which says for a smooth complex variety $Y$ of dimension $n$, and an ample divisor $A$ on $Y$, we have 
\[
    H^q\big(Y, \Omega^p_Y \otimes \O_Y(A)\big) = 0 
\]
if $p+q>n$. By "borderline Nakano vanishing", we refer to the vanishings of "boundary" terms $H^q\big(Y, \Omega_Y^{n - q} \otimes \O_Y(A)\big)$. In section \ref{sec Borderline}, we will prove some borderline Nakano vanishings for hypersurfaces in $\PP^n$. If $n \ge 4$, we obtain

\begin{thm} \label{Bott vanishing hypersurfaces}
    Suppose $X$ is a smooth hypersurface in $\PP^n$ of degree $d \ge 2$ with $n \ge 4$. Then
    \begin{enumerate}
        \item If $i+p\neq n-1$, we have $H^i(X, \Omega^p_X \otimes \O_X(l))=0$ unless 
        \begin{itemize}
            \item $i=p$ and $l=0$,
            \item $i=0$ and $l>p$, or
            \item $i=n-1$ and $l<p-n+1$.
        \end{itemize}
        \item If $i+p=n-1$, then 
        \begin{itemize}
            \item For $0<i<n-1$, we have  $H^{i}\big(X,\Omega^p_X \otimes \O_X(l)\big) = 0$
            if $l \ge (p+1)d-n$ or $l\le (p-n)d + n$.
            \item  $H^0\big(X, \Omega^{n-1}_X \otimes \O_X(l)\big)=0$ if and only if $l\le n-d$.
            \item  $H^{n-1}\big(X, \O_X(l)\big)=0$ if and only if $l\ge d-n$.
        \end{itemize}
    \end{enumerate}
\end{thm}

When $X$ is a hypersurface of $\PP^3$ of degree $d$, the Picard group of $X$ can be large. In this case, we prove a weaker vanishing theorem

\begin{prop} \label{bott vanishing surfaces}
    Let $X$ be a smooth hypersurface in $\PP^3$ of degree $d\ge 2$. Then:
    \[
        H^1(X, \Omega_X^1 \otimes L) = 0 
    \]
    if $L(4-2d)$ is ample. 
\end{prop}

\begin{remark}
    In section \ref{sec Borderline}, by using Macaulay's theorem \cite[Theorem 7.4.1]{periodmappinganddomain}, we show that the partial derivatives of the polynomial defining $X$ generate all polynomials of degree at least $(n+1)(d-2)+1$. This in turn gives
    \[H^{n-2}\big(X, \Omega_X^1\otimes \O_X((p+1)d-n-1)\big) \neq 0\]
    which says that our bound is where the first non-vanishing appears.
\end{remark}
\vspace{\bigskipamount}

\noindent \textbf{Acknowledgements.}
I would like to express my sincere gratitude to my advisor Mihnea Popa for suggesting the problem, for helpful discussions, and for his constant support during the preparation of this paper. I thank Wanchun Shen and Jit Wu Yap for their useful comments on the draft of this paper.

\vspace{\bigskipamount}

\section{Preliminaries}

\subsection{Operations on exact sequences}

We first discuss the general theory behind the construction of a Koszul-type resolution that will be used throughout the proof of our main results. In what follows, by a \textit{variety} we mean a reduced separated scheme of finite type over $\CC$. Consider an exact sequence of vector bundles on a variety $X$:
\[
    0 \rightarrow A \rightarrow B \rightarrow C \rightarrow 0.
\]
If $C$ is a line bundle, it is well-known that we can form the following short exact sequence for any positive integer $p$: 
\begin{equation} \label{wedge_seq}
    0 \rightarrow \bigwedge^pA \rightarrow \bigwedge^pB \rightarrow \bigwedge^{p-1}A \otimes C \rightarrow 0.
\end{equation}
Dually, if $A$ is a line bundle, we get instead: 
\begin{equation} \label{wedge_seq_2}
    0 \rightarrow A \otimes \bigwedge^{p-1}C \rightarrow \bigwedge^pB \rightarrow \bigwedge^pC \rightarrow 0.
\end{equation}

When $X$ is a smooth hypersurface of $\PP^n$, we apply the above operations to the restriction of the Euler sequence on $\PP^n$ to $X$
\[
    0 \rightarrow \Omega^1_{\PP^n}|_X \rightarrow \bigoplus_{n+1} \O_X(-1) \rightarrow \O_X \rightarrow 0
\]
to obtain a resolution for $\Omega^p_{\PP^n}|_X$:

\begin{sublem} \label{resol_Omega_p}
    Let $X$ be a smooth hypersurface of $\PP^n$, then $\Omega^p_{\PP^n}|_X$ admits the following Koszul-type resolution:
    \[
        0 \rightarrow \O_X(-n-1) \rightarrow \bigoplus_{n+1} \O_X(-n) \rightarrow \dots \rightarrow \bigoplus_{\binom{n+1}{p+1}} \O_X(-p-1) \rightarrow \Omega^p_{\PP^n}|_X \rightarrow 0.
    \]
    which is the restriction of the familiar resolution of $\Omega_{\PP^n}^p$ on $\PP^n$ to $X$. 
\end{sublem}


\subsection{Hodge ideals}

In this section, we briefly recall the definition and some properties of Hodge ideals. Let $X$ be a smooth complex variety of dimension $n$. We denote by $\mathscr{D}_X$ the sheaf of differential operators on $X$. A filtered left $\mathscr{D}_X$-module $(\M, F)$ is a left $\mathscr{D}_X$-module $\M$ equipped with an increasing filtration $F_{\bullet}\M$ by coherent $\O_X$-modules, bounded below and satisfying 
\[
    F_k \mathscr{D}_X . F_l \M \subseteq F_{k+l} \M,
\]
where $F_k \mathscr{D}_X$ is the sheaf of differential operators of order $\le k$. To a filtered left $\mathscr{D}_X$-module $(\M, F)$, one can associate the de Rham complex
\[
    \operatorname{DR}(\M,F) = [0 \rightarrow \M
    \rightarrow \Omega_X^1 \otimes \M
    \rightarrow \dots \rightarrow \omega_X \otimes \M]
\]
with the filtration
\[
    F_{k}\operatorname{DR}(\M,F) = [0 \rightarrow F_k\M
    \rightarrow \Omega_X^1 \otimes F_{k+1}\M
    \rightarrow \dots \rightarrow \omega_X \otimes F_{k+n}\M]
\]
placed in degree $-n,\dots, 0$. The associated graded complex is 
\[
    \operatorname{gr}^F_k\operatorname{DR}(\M,F) = [0 \rightarrow \operatorname{gr}^F_k\M
    \rightarrow \Omega_X^1 \otimes \operatorname{gr}^F_{k+1}\M
    \rightarrow \dots \rightarrow \omega_X \otimes \operatorname{gr}^F_{k+n}\M].
\]
There is a standard correspondence between filtered left and right $\mathscr{D}$-modules given by 
\[
    \M \mapsto \M \otimes_{\O_X} \omega_X = \M^r 
\]
with the corresponding filtration $F_{p-n}(\M^r) \cong F_p\M \otimes \omega_X$. The associated de Rham complex can be written in terms of the filtered right $\mathscr{D}_X$-module $(\M^r, F)$ as
\[
    \operatorname{gr}^F_k\operatorname{DR}(\M^r,F) = [0 \rightarrow \operatorname \bigwedge^{n} T_X \otimes {gr}^F_{k-n} \M^r
    \rightarrow \bigwedge^{n-1} T_X \otimes \operatorname{gr}^F_{k-n+1}\M^r
    \rightarrow \dots \rightarrow \operatorname{gr}^F_{k}\M^r].
\]

\vspace{5mm}

For a reduced effective divisor $D$ on $X$, the $\mathscr{D}_X$-module $\O_X(*D) \coloneqq \bigcup_{k\ge 0} \O_X(kD)$ underlies the mixed Hodge module $j_{*}\mathbb{Q}_U^{H}[n]$ \cite[Proposition 2.11]{Saito_MHM}, where $j:U\coloneqq X\setminus D \hookrightarrow X$, and hence is equipped with the Hodge filtration $F$. Moreover, $\O_X(*D)$ has the natural pole order filtration on it. Musta\c t\u a and Popa \cite{Hodge_ideals} introduce the notion of Hodge ideals $I_k(D)$ to measure the discrepancy between these two filtrations $F_k(\O_X(*D)) = I_k(D) \otimes \O_X((k+1)D)$. In \cite{Hodge_ideals_Q1} and \cite{Hodge_ideals_Q2}, the authors extend the definition of Hodge ideals to an arbitrary effective $\QQ$-divisor $D$ on $X$. Locally, we can assume that $D=\alpha \cdot \text{div}(h)$ for some nonzero regular function $h$ and some positive rational number $\alpha$. Denote by $Z$ the support of $D$. Consider the left $\mathscr{D}_X$-module $\O_X(*Z)h^{-\alpha}$, the rank $1$ free $\O_X(*Z)$-module with generator the symbol $h^{-\alpha}$, on which a derivation $D$ of $\O_X$ acts by 
\[
    D(wh^{-\alpha}) = \big ( D(w) - \alpha w\dfrac{D(h)}{h}\big)h^{-\alpha}.
\]
Although $O_X(*Z)h^{-\alpha}$ does not necessarily underlie a mixed Hodge module, by \cite[Lemma 2.6]{Hodge_ideals_Q1}, it is a direct summand of a filtered $\mathscr{D}_X$-module underlying a mixed Hodge module. This induces a canonical Hodge filtration on $\O_X(*Z)h^{-\alpha}$, which makes it a filtered $\mathscr{D}_X$-module. The \emph{Hodge ideals} $I_k(D)$ are then defined as 
\[
    F_k(\O_X(*Z)h^{-\alpha}) = I_k(D) \otimes_{\O_X} \O_X(kZ)h^{-\alpha}.
\]
This definition is, in fact, independent of the choice of $h$ and $\alpha$, and hence the Hodge ideals can be defined globally on $X$.

\vspace{\medskipamount} 

The Hodge ideals play an important role in the study of the singularities of the divisor $D$. For example, it is not difficult to see that $I_0(X) = \mathcal{J}((1-\epsilon)D)$, the multiplier ideal associated to the $\mathbb{Q}$-divisor $(1-\epsilon)D$ for any $0 < \epsilon \ll 1$ \cite[Proposition 9.1]{Hodge_ideals_Q1}. From the properties of multiplier ideals \cite[Definition 9.3.9]{positivity2}, the pair $(X, D)$ is log canonical if and only if $I_0(D)$ is trivial. If we go one step further, in the case $D$ is integral, it can be shown that the Hodge ideal $I_1(D)$ is contained in the adjoint ideal $\operatorname{adj}(D)$, see \cite[Theorem C]{Hodge_ideals}. Hence, by \cite[Theorem 9.3.48]{positivity2}, $D$ is normal and has at worst rational singularities if $I_1(D)$ is trivial. Furthermore, Musta\c t\u a and Popa prove that if $D$ is of the form $D=\alpha Z$, where $Z$ is a reduced divisor and $0 < \alpha \le 1$, then $Z$ is smooth if and only if all Hodge ideals $I_k(D)$ are trivial \cite[Corollary 11.12]{Hodge_ideals_Q1}. Thus, the Hodge ideals of $D = \alpha Z$ are cosupported inside the singular locus of $Z$ where the cosupport of an ideal $I$ is the support of $\O_X/I$. We denote by $Z_k$ the closed subscheme of $X$ defined by the ideal sheaf $I_k(D)$. 

\vspace{\medskipamount}

We are interested in a vanishing theorem for Hodge ideals. A combination of \cite[Theorem F]{Hodge_ideals} (in the case $D$ is a reduced divisor), \cite[Theorem 12.1]{Hodge_ideals_Q1} and \cite[Theorem 1.1]{BingYi} gives the following theorem: 
\begin{subthm}
    Let $X$ be a smooth projective variety of dimension $n$, $D$ an effective divisor, and $L$ a line bundle on $X$ such that $L(-D)$ is ample. For some $k \ge 1$, assume that the pair $(X,D)$ is reduced $(k-1)$-log canonical i.e $I_j(D) = \O_X(Z-\cD)$ for all $0 \le j \le k-1$. Then we have:
\begin{enumerate}
    \item If $k \le n$, and $L(pZ-\cD)$ is ample for all $1 \le p \le k$, then 
    \[
        H^i\big(X, \omega_X(kZ) \otimes I_k(D) \otimes L\big) = 0
    \]
    for all $i \ge 2$. Moreover,  
    \[
        H^1\big(X, \omega_X(kZ) \otimes I_k(D) \otimes L\big) = 0 
    \]
    holds if $H^j\big(X, \Omega^{n-j} \otimes L((k-j+1)Z-\cD)\big) = 0$ for all $1 \le j \le k$.
    
    \item If $k \ge n+1$, then $Z$ is smooth, and so $I_k(D) = \O_X(Z-\cD)$. In this case, if $L(kZ-\cD)$ is ample, then 
    \[
        H^i\big(X, \omega_X(kZ) \otimes I_k(D) \otimes L\big) = 0
    \]
    for all $i \ge 1$
    
    \item If $D$ is ample, then $(1)$ and $(2)$ also hold with $L(-D)$ nef. 
\end{enumerate}
\end{subthm}

Unlike the Hodge ideals, the module $\O_X(*Z)h^{-\alpha}$ cannot be defined globally. To guarantee the existence of a global filtered $\mathscr{D}_X$-module, which is locally isomorphic to $(\O_X(*Z)h^{-\alpha},F)$, we need the extra assumption that there exists a line bundle $M$ such that $M^{\otimes l}\cong \O_X(lD)$ as integral divisors. In \cite{BingYi}, Chen shows that we can remove this global assumption by constructing a global sheaf
\[
    \M_k(D) \coloneqq \dfrac{\omega_X(kZ)\otimes_{\O_X} I_k(D)}{\omega_X((k-1)Z) \otimes_{\O_X} I_{k-1}(D)},
\]
which is locally isomorphic to $\omega_X \otimes gr^F_k(\O_X(*Z)h^{-\alpha})$ for any $k\in \ZZ$. The $k$-th Spencer complex of $\M_\bullet(D)$
\[
    Sp_k(\M_\bullet(D)) \coloneqq [\M_k(D) \otimes_{\O_X} \bigwedge^n T_X \rightarrow \M_{k+1}(D) \otimes_{\O_X} \bigwedge^{n-1}T_X \rightarrow \dots \rightarrow \M_{k+n}(D)]
\]
placed in degree $-n,\dots,0$ is locally isomorphic to $gr^F_kDR(\O_X(*Z)h^{-\alpha})$. This complex enjoys a similar vanishing theorem as Saito's vanishing theorem for mixed Hodge modules:

\begin{subthm} (\cite[Theorem 1.2]{BingYi}) \label{Saito_type_vanishing}
    Let $X$ be a smooth complex projective variety of dimension $n$, $D$ an effective $\QQ$-divisor on $X$, and $L$ a line bundle on $X$. Then we have: 
    \begin{enumerate}
        \item If $L-D$ is ample, then 
        \[
            \HH^i\big(X, Sp_k(\M_\bullet(D)) \otimes_{\O_X} L\big) = 0
        \]
        for any $i>0$ and any integer $k$.
        
        \item If $D$ is ample, then $(1)$ also holds for $L-D$ nef. 
    \end{enumerate}
\end{subthm}

\noindent This theorem plays a crucial role in the proof of the vanishing theorem for Hodge ideals of $\QQ$-divisors. 

\vspace{\bigskipamount}


\section{Borderline Nakano vanishing for hypersurfaces in $\PP^n$} \label{sec Borderline}


Because of the well-known Bott vanishing for projective spaces which says that the cohomology groups $H^i(\PP^n, \Omega^p_{\PP^n} \otimes \O_{\PP^n}(l)) = 0$ unless
\begin{itemize}
    \item $i=p$ and $l=0$,
    \item $i=0$ and $l>p$, or 
    \item $i=n$ and $l< p-n$, 
\end{itemize}
see for example \cite[Theorem 7.2.3]{periodmappinganddomain},
we only consider hypersurfaces of degree at least $2$. In this section, $X$ will denote a smooth hypersurface of degree $d \ge 2$ of $\PP^n$, and $L$ a line bundle on $X$ unless otherwise noted. 

\subsection{The case $n \ge 4$} First, we start with a vanishing theorem for the sheaves $\Omega^p_{\PP^n}|_X$. 

\begin{sublem} \label{bott vanishing res of Pn}
    Let $X$ be a smooth hypersurface in $\PP^n$ of degree $d\ge 2$ with $n \ge 4$. We have 
    \[H^i(X, \Omega_{\PP^n}^p|_X \otimes \O_X(l))=0\] unless  
    \begin{itemize}
        \item $i=p$ and $l=0$, or
        \item $i=p-1$ and $l=d$, or
        \item $i=0$ and $l>p$, or
        \item $i=n-1$ and $l<d+p-n$.
    \end{itemize}   
\end{sublem}

\begin{proof}

Consider the following short exact sequence:
\[0\rightarrow \Omega_{\PP^n}^p\otimes \O_{\PP^n}(l-d) \rightarrow \Omega_{\PP^n}^p \otimes \O_{\PP^n}(l) \rightarrow \Omega_{\PP^n}^p|_X \otimes \O_X(l) \rightarrow 0.\]
By considering its associated long exact sequence and the Bott vanishing theorem for $\PP^n$, we easily see that for $0<i<n-1$
\[H^i\big(X, \Omega_{\PP^n}^p|_X \otimes \O_X(l)\big) = 0\]
unless $i=p$ and $l=0$, or $i=p-1$ and $l=d$. Now, we are left with two remaining cases $i=0$ and $i=n-1$. By duality, we only need to consider the case $i=0$. 

If $p=0$, it is clear that $H^0\big(X, \O_X(l)\big) = 0$ if $l<0$. If $p\ge 1$, we need to show that 
\[H^0(X, \Omega_{\PP^n}^p|_X \otimes \O_X(l))=0\]
for $l\le p$. Consider the exact sequence
\[0\rightarrow H^0\big(\PP^n, \Omega_{\PP^n}^p \otimes \O_{\PP^n}(l-d)\big) \rightarrow H^0\big(\PP^n, \Omega_{\PP^n}^p \otimes \O_{\PP^n}(l)\big) \rightarrow H^0\big(X, \Omega_{\PP^n}^p|_X \otimes \O_{X}(l)\big) \rightarrow H^1\big(\PP^n, \Omega_{\PP^n}^p \otimes \O_{\PP^n}(l-d)\big).\]
The Bott vanishing theorem for projective spaces gives 
\[H^0(\PP^n, \Omega_{\PP^n}^p\otimes \O_{\PP^n}(l)) = H^1(\PP^n, \Omega^p_{\PP^n}\otimes \O_{\PP^n}(l-d)) = 0\]
which implies 
\[H^0(X, \Omega^p_{\PP^n}|_X \otimes \O_X(l)) = 0.\]

\end{proof}

\vspace{\medskipamount}

Now, to obtain borderline Nakano, as well as further vanishings for the hypersurface $X$, we consider the conormal exact sequence
\[
    0 \rightarrow \O_X(-d) \rightarrow \Omega^1_{\PP^n}|_X \rightarrow \Omega^1_X \rightarrow 0.
\]
By using (\ref{wedge_seq}), we get the short exact sequence
\[
    0 \rightarrow \Omega^{p-1}_X \otimes \O_X(l-d) 
    \rightarrow \Omega^p_{\PP^n}|_X \otimes \O_X(l) 
    \rightarrow \Omega^p_X \otimes \O_X(l)
    \rightarrow 0.
\] 
The associated long exact sequence in cohomology is 

\begin{flushleft}
    $\dots \rightarrow H^{i}\big(X, \Omega_{\PP^n}^p|_X \otimes \O_X(l)\big) \rightarrow H^{i}\big(X,\Omega^p_X \otimes \O_X(l)\big) \rightarrow$
\end{flushleft}
    
\begin{flushright}
    $\rightarrow H^{i+1}\big(X, \Omega^{p-1}_X \otimes \O_X(l-d)\big) \rightarrow H^{i+1}\big(X, \Omega_{\PP^n}^p|_X \otimes \O_X(l)\big) \rightarrow \dots$
\end{flushright}

\noindent Our strategy is to use Lemma \ref{bott vanishing res of Pn} on the above long exact sequence to reduce to extreme cases. 
\begin{itemize}
    \item If $n-3\ge i>p$, we get
    \[H^i\big(X, \Omega_X^p\otimes \O_X(l)\big) \cong H^{i+1}\big(X,\Omega_X^{p-1}\otimes \O_X(l-d)\big).\]
    \item If $2\le i<p$, we get
    \[H^i\big(X, \Omega_X^p\otimes \O_X(l)\big) \cong H^{i-1}\big(X,\Omega_X^{p+1}\otimes \O_X(l+d)\big).\]
    \item If $i=p$, we have the exact sequence
    \[H^{p}\big(X, \Omega_{\PP^n}^p|_X \otimes \O_X(l)\big) \rightarrow H^{p}\big(X,\Omega^p_X \otimes \O_X(l)\big) \rightarrow H^{p+1}\big(X, \Omega^{p-1}_X \otimes \O_X(l-d)\big) \rightarrow 0.\]
\end{itemize}

\vspace{\medskipamount}

\begin{proof}[Proof of Theorem \ref{Bott vanishing hypersurfaces}]

Part $(1)$: $i+p\neq n-1$. By Serre duality, we can assume $i+p<n-1$. 

\begin{itemize}
    \item If $i>p$, then by the above argument, we obtain
\[H^i(X, \Omega^p_X\otimes \O_X(l)) \cong H^{i+1}(X, \Omega^{p-1}_X\otimes \O_X(l-d)) \cong \dots \cong H^{i+p}(X, \O_X(l-pd)) = 0.\]
    \item If $i=p>0$, by Lemma \ref{bott vanishing res of Pn}, we have 
\[H^p(X, \Omega^p_X\otimes \O_X(l)) \cong H^{p+1}(X, \Omega^{p-1}_X\otimes \O_X(l-d)) = 0\] 
for any $l\neq 0$.
    \item If $0<i<p$, we need the following claim:
    
    \textit{Claim:} The natural morphism 
    \[H^k(X, \Omega^k_X) \rightarrow H^k(X, \Omega^{k+1}_{\PP^n}|_X \otimes \O_X(d))\] 
    is an isomorphism for all $k$ such that $2k < n-1$. 

    This claim follows from the following commutative diagram:
    \[\begin{tikzcd}
	{H^k(X, \Omega_X^k) } & {H^k(X, \Omega_{\PP^n}^{k+1}|_X (d)) } \\
	{H^k(X, \Omega_{\PP^n}^{k}|_X) } \\
	{H^{k}(\PP^n, \Omega^{k}_{\PP^n})} & {H^{k+1}(\PP^n, \Omega^{k+1}_{\PP^n})}
	\arrow[from=1-1, to=1-2]
	\arrow["\cong", from=1-2, to=3-2]
	\arrow["\cong", from=2-1, to=1-1]
	\arrow["\cong", from=3-1, to=2-1]
	\arrow["\cong"', from=3-1, to=3-2]
    \end{tikzcd}\]
    in which the lower horizontal morphism is the Lefschetz morphism. Moreover, the isomorphisms 
    \[H^k(X, \Omega^{k+1}_{\PP^n}|_X(d)) \xrightarrow{\cong} H^{k+1}(\PP^n, \Omega^{k+1}_{\PP^n})\]
    and 
    \[H^{k}(\PP^n, \Omega^{k}_{\PP^n}) \xrightarrow{\cong} H^k(X, \Omega^{k}_{\PP^n}|_X)\]
    are the consequences of the Bott vanishing for projective spaces, while the isomorphism 
    \[H^k(X, \Omega^{k}_{\PP^n}|_X) \xrightarrow{\cong} H^k(X, \Omega^k_X)\]
    is due to the first bullet point, i.e. $H^k(X, \Omega^{k-1}_X(-d) = H^{k+1}(X, \Omega^{k-1}_X(-d)) = 0$. The claim is proven. 
    
    Now, we consider two cases:
    
    Case 1: $i+p = 2k+1$, where $0<k<n-1$ by the assumption. By the above argument, we get 
    \[H^i\big(X, \Omega_X^p \otimes \O_X(l)\big) \cong \dots \cong H^k\big(X, \Omega^{k+1}_X \otimes \O_X(l-(k-i)d)\big)\]
    and the exact sequence
    \[H^k\big(X, \Omega^k_X \otimes \O_X(l-(k-i+1)d)\big) \rightarrow H^k\big(X, \Omega^{k+1}_{\PP^n}|_X \otimes \O_X(l-(k-i)d)\big) \rightarrow H^k\big(X, \Omega^{k+1}_X \otimes \O_X(l-(k-i)d)\big) \rightarrow 0.\]
    If $l-(k-i)d \neq d$, the vanishing follows from Lemma \ref{bott vanishing res of Pn}. Otherwise, the first morphism is an isomorphism, which still implies $H^k(X, \Omega_X^{k+1} \otimes \O_X(d)) = 0$. 

    Case 2: $i+p = 2k$, where $1<k<n-1$ by the assumption. Similarly, we get 
    \[H^i\big(X, \Omega_X^p \otimes \O_X(l)\big) \cong \dots \cong H^{k-1}\big(X, \Omega^{k+1}_X \otimes \O_X(l-(k-i-1)d)\big)\]
    and the exact sequence
    \[0 \rightarrow H^{k-1}\big(X, \Omega^{k+1}_X \otimes \O_X(l-(k-i-1)d)\big) \rightarrow  H^k\big(X, \Omega^{k}_X \otimes \O_X(l-(k-i)d)\big) \rightarrow H^{k}\big(X, \Omega^{k+1}_{\PP^n}|_X \otimes \O_X(l-(k-i-1)d)\big)\]
    which gives the vanishing of $H^{k-1}\big(X, \Omega^{k+1}_X \otimes \O_X(l-(k-i-1)d)\big)$. 

    \item If $i=0$, consider the exact sequence
    \[H^{0}(X, \Omega^{p}_{\PP^n}|_X \otimes \O_X(l)) \rightarrow H^{0}(X, \Omega^{p}_X \otimes \O_X(l)) \rightarrow H^{1}(X, \Omega^{p-1}_X \otimes \O_X(l-d))
    \rightarrow H^{1}(X, \Omega^{p}_{\PP^n}|_X \otimes \O_X(l)).\]
    The first term vanishes if $l\le p$ by Lemma \ref{bott vanishing res of Pn}. The third term equals to $0$ unless $p=2$ and $l=d$, in which case, the last morphism is an isomorphism. Thus 
    \[H^{0}(X, \Omega^{p}_X \otimes \O_X(l)) = 0 \quad \text{ if } \quad l\le p.\]
    Moreover, for any $l> p$, 
    \[H^{0}(X, \Omega^{p}_{\PP^n}|_X \otimes \O_X(p+1))\hookrightarrow H^{0}(X, \Omega^{p}_{\PP^n}|_X \otimes \O_X(l))\]
    and 
    \[0\rightarrow H^{0}(X, \Omega^{p}_{\PP^n}|_X \otimes \O_X(p+1)) \rightarrow H^{0}(X, \Omega^{p}_X \otimes \O_X(p+1))\] 
    imply that 
    \[H^{0}(X, \Omega^{p}_X \otimes \O_X(l)) \neq 0 \quad \text{ if } \quad l > p.\]
\end{itemize}
The proof of part $(1)$ is complete. 

\vspace{\medskipamount}

Part $(2)$: $i+p = n-1$. The case $i=0$ and $i=n-1$ are obvious. For the remaining parts, by the argument before this proof, it suffices to prove this theorem for the special cases $p = 1$ and $p=n-2$:
\[
    H^{n-2}\big(X, \Omega^1_X \otimes \O_X(l)\big) = 0 \quad  \text{ if } \quad  l \ge 2d-n,  
\]
and 
\[
    H^{1}\big(X, \Omega^{n-2}_X \otimes \O_X(l)\big) = 0 \quad  \text{ if } \quad  l \ge (n-1)d-n.
\]
Notice that if $i = p = \dfrac{n-1}{2}$, then $l\ge (p+1)d-n > 0$ for $d\ge 2$; hence $H^p\big(X, \Omega^p_X \otimes \O_X(l)\big) \cong H^{p+1}\big(X, \Omega^{p-1}_X \otimes \O_X(l-d)\big)$. 

\noindent To show the first statement, consider the exact sequence
\[
    0 \rightarrow H^{n-2}\big(X, \Omega^1_X \otimes \O_X(l)\big) 
    \rightarrow H^{n-1}\big(X, \O_X(l-d)\big)
    \rightarrow H^{n-1}\big(X, \Omega^1_{\PP^n}|_X \otimes \O_X(l)\big)
    \rightarrow  \dots,
\]
where the vanishing on the left is due to Lemma \ref{bott vanishing res of Pn}. The last term $H^{n-1}\big(X, \Omega^1_{\PP^n}|_X \otimes \O_X(l)\big)$ sits in the following exact sequence \[
    0 \rightarrow H^{n-1}\big(X, \Omega^1_{\PP^n}|_X \otimes \O_X(l)\big)
    \rightarrow V \otimes H^{n-1}\big(X, \O_X(l-1)\big) 
    \rightarrow H^{n-1}\big(X, \O_X(l)\big)
    \rightarrow 0,
\]
where $V = H^0\big(X, \O_X(1)\big)$. By Serre duality, 
\[
    H^{n-1}\big(X, \O_X(l-d)\big) \cong H^0\big(X, \O_X(2d-l-n-1)\big)^\vee
\]
and 
\[
    H^{n-1}\big(X, \Omega^1_{\PP^n}|_X \otimes \O_X(l)\big) \cong \operatorname{coker}\Big(H^0\big(X, \O_X(d-l-n-1)\big) \rightarrow V^\vee \otimes H^0\big(X, \O_X(d-l-n)\big)\Big)^\vee.
\]
These isomorphisms imply that the vanishing of $H^{n-2}\big(X, \Omega^1_X \otimes \O_X (l)\big)$ is equivalent to the surjectivity of the morphism 
\[
\operatorname{coker}\Big(H^0\big(X, \O_X(d-l-n-1)\big) \rightarrow V^\vee \otimes H^0\big(X, \O_X(d-l-n)\big)\Big) \rightarrow H^0\big(X, \O_X(2d-l-n-1)\big).
\]
If $l \ge 2d-n$, then both sides vanish; the first statement is proved. 

\noindent If $2d-n-1 \ge l > d-n$, then the morphism is simply $0 \rightarrow H^0\big(X, \O_X(2d-s-n-1)\big) \neq 0$ which is never surjective.

\vspace{\medskipamount}

Set $r = l+d-n-1$. The second statement is then equivalent to 
\[
    H^1\big(X, T_X \otimes \O_X(r)\big)=0 \quad  \text{ if } \quad  r \ge n(d - 2)-1.
\]
The normal sequence and the Euler sequence together form the diagram:
\[\begin{tikzcd}
	& 0 & 0 & 0 \\
	0 & {T_X} & {T_{\PP^n}|_X} & {N_{X/\PP^n}} & 0 \\
	0 & {\ker(J_X(f))} & {V^\vee\otimes\O_X(1)} & {N_{X/\PP^n}} & 0 \\
	& {\O_X} & {\O_X} \\
	& 0 & 0
	\arrow[from=2-1, to=2-2]
	\arrow[from=2-2, to=2-3]
	\arrow[from=2-3, to=2-4]
	\arrow[from=2-4, to=2-5]
	\arrow[from=2-2, to=1-2]
	\arrow[from=2-3, to=1-3]
	\arrow[from=2-4, to=1-4]
	\arrow[from=3-2, to=2-2]
	\arrow[from=3-3, to=2-3]
	\arrow[equal, from=3-4, to=2-4]
	\arrow[from=4-2, to=3-2]
	\arrow[from=4-3, to=3-3]
	\arrow[from=5-2, to=4-2]
	\arrow[from=5-3, to=4-3]
	\arrow[from=3-2, to=3-3]
	\arrow[equal, from=4-2, to=4-3]
	\arrow["{J_X(f)}", from=3-3, to=3-4]
	\arrow[from=3-1, to=3-2]
	\arrow[from=3-4, to=3-5]
\end{tikzcd}\]
where $f$ is a homogeneous polynomial of degree $d$ defining $X$, and $J_X(f)$ is the multiplication by the Jacobian of $f$ which sends a $(n+1)$-tuple of linear functionals $(g_0,\dots,g_n)$ to $\sum g_i \dfrac{\partial f}{\partial x_i}$. By tensoring with $\O_X(r)$ and taking the long exact sequence of cohomology, we get
\[\begin{tikzcd}
	{H^1\big(X, \O_X(r)\big) = 0} \\
	{H^0\big(X, T_{\PP^n}|_X \otimes \O_X(r)\big)} & {H^0\big(X, \O_X(d+r)\big)} & {H^1\big(X, T_X \otimes \O_X(r)\big)} & {0} \\
	{V^\vee \otimes H^0\big(X, \O_X(r+1)\big)} & {H^0\big(X, \O_X(d+r)\big)}
	\arrow["\phi", from=2-1, to=2-2]
	\arrow[two heads, from=2-2, to=2-3]
	\arrow[from=2-3, to=2-4]
	\arrow[from=2-1, to=1-1]
	\arrow[two heads, from=3-1, to=2-1]
	\arrow[equal, from=3-2, to=2-2]
	\arrow["{J_X(f)}", from=3-1, to=3-2]
\end{tikzcd}\]
where the vanishing on the right is due to Lemma \ref{bott vanishing res of Pn}. Thus, the vanishing of $H^1\big(X, T_X \otimes \O_X (r)\big)$ is equivalent to the surjectivity of $\phi$, and hence of the morphism $J_X(f)$.

\vspace{\medskipamount}

Consider the following diagram:
\[\begin{tikzcd}
	{V^\vee \otimes H^0\big(X, \O_X(r+1)\big)} & {H^0\big(X, \O_X(d+r)\big)} \\
	{V^\vee \otimes H^0\big(\PP^n, \O_{\PP^n}(r+1)\big)} & {H^0\big(\PP^n, \O_{\PP^n}(d+r)\big)} \\
	& {H^0\big(\PP^n, \O_{\PP^n}(r)\big)}
	\arrow["{J_X(f)}", from=1-1, to=1-2]
	\arrow["{J_{\PP^n}(f)}", from=2-1, to=2-2]
	\arrow[two heads, from=2-1, to=1-1]
	\arrow[two heads, from=2-2, to=1-2]
	\arrow[dashed, from=2-1, to=1-2]
	\arrow["f", hook, from=3-2, to=2-2]
\end{tikzcd}\]
Notice that the vertical inclusion is given by the multiplication by $f$ (that we also denote by $f$). Since $d\cdot f = \sum_{i=0}^{n} x_i \dfrac{\partial f}{\partial x_i}$, the image of $J_{\PP^n}(f)$ contains the image of $f$, which means the surjectivities of $J_X(f)$ and $J_{\PP^n}(f)$ are equivalent. Combining these arguments, it suffices to prove that $J_{\PP^n}(f)$ is surjective if $r \ge n(d-2) - 1$. But this follows from the following lemma; the proof of Theorem \ref{Bott vanishing hypersurfaces} is complete.
\end{proof} 

\begin{sublem} [Macaulay's theorem] \emph{(\cite[Theorem 7.4.1]{periodmappinganddomain})}
    Suppose there is an ordered $(n+1)$-tuple of homogeneous polynomials $F_i \in \mathbb{C}[x_0,\dots,x_n]$ of degree $d_i$ that form a regular sequence. Let $R = \mathbb{C}[x_0,\dots,x_n]/(F_0,\dots, F_n)$ be the quotient of the polynomial ring by the ideal generated by $\{F_i\}_{i=0}^{n}$. Then:
    \begin{enumerate}
        \item $R_k = 0$ if $k > \rho = \sum_{i=0}^{n} d_i - n - 1$ and $R_\rho \cong \mathbb{C}$.
        \item The multiplication in the polynomial ring induces a perfect pairing $R_k \otimes R_{\rho -k} \rightarrow R_\rho$.
    \end{enumerate}
\end{sublem}

\noindent We use the well-known fact that the partial derivatives of a homogeneous polynomial giving a smooth hypersurface in $\PP^n$ form a regular sequence in $\CC[x_0, \dots, x_n]$. 

\vspace{\medskipamount}

\begin{remark}
    Although we do not know whether there is extra borderline Nakano vanishing, the lower bound for $l$ in Theorem \ref{Bott vanishing hypersurfaces} cannot be made smaller. As we have seen in the proof of the first statement, 
    \[H^{n-2}\big(X, \Omega^1_X \otimes \O_X(l)\big) \neq 0 \quad if \quad 2d-n-1 \ge l > d-n.\]
    On the other hand, for $l = (n-1)d-n-1$, part $(2)$ of Macaulay's theorem implies that $J_X(f)$ is not surjective, or equivalently, 
    \[H^1\big(X, \Omega^{n-2}_X \otimes \O_X((n-1)d-n-1)\big) \neq 0.\] 
\end{remark}

\vspace{\bigskipamount}

\subsection{The case $n=3$}

\begin{proof}[Proof of Proposition \ref{bott vanishing surfaces}] We will consider the cases $d=2$, $d=3$ and $d\ge 4$ separately. 

\textbf{Case 1}: If $d=2$, $X$ is a quadric surface, which is isomorphic to $\PP^1 \times \PP^1$ via the Segre embedding into $\PP^3$.  Any ample line bundle on $X$ is of the form $\O_X(a,b) \coloneqq p_1^*\O_{\PP^1}(a) \otimes p_2^*\O_{\PP^1}(b)$ with $a,b>0$, where $p_1, p_2:X \rightarrow \PP^1$ are the projections to $\PP^1$. Consider the exact sequence of sheaves of relative differential forms:
\[\begin{tikzcd}
	0 & {p_2^*\Omega^1_{\PP^1}} & {\Omega^1_X} & {\Omega^1_{X/\PP^1}} & \bullet \\
	& {\O_X(0,-2)} && {i_*\O_{\PP^1}(-2)}
	\arrow[from=1-1, to=1-2]
	\arrow[from=1-2, to=1-3]
	\arrow[from=1-3, to=1-4]
	\arrow[from=1-4, to=1-5]
	\arrow[equal, from=2-2, to=1-2]
	\arrow[equal, from=2-4, to=1-4]
\end{tikzcd}\]
where $i: \PP^1 \rightarrow X$ is the embedding of the first factor. Since $\omega_X \cong \O_X(-2,-2)$, then 
\[H^1\big(X, O_X(0,-2) \otimes \O_X(a,b)\big) \cong H^1\big(X, \omega_X \otimes O_X(a+2,b)\big) = 0\]
and 
\[H^1\big(X, i_*\O_{\PP^1}(-2) \otimes \O_X(a,b)\big) \cong H^1\big(\PP^1, \O_{\PP^1}(a-2)\big) = 0.\]
As a result, we have 
\[H^1\big(X, \Omega^1_X \otimes \O_X(a,b)\big) = 0.\]

\textbf{Case 2}: If $d=3$, $X$ is a smooth cubic surface, which is a blow-up of $\PP^2$ at $6$ points. Let $\pi:X \rightarrow \PP^2$ be the blow-up morphism with exceptional divisors $E_1, \dots, E_6$ isomorphic to $\PP^1$. Consider the exact sequence:
\[
    0 \rightarrow \pi^*\Omega^1_{\PP^2} \rightarrow \Omega^1_{X} \rightarrow i_*\omega_{E/\PP^2} \rightarrow 0,
\]
where $i:E = \bigsqcup E_i \rightarrow X$ is the inclusion of the exceptional divisor. Taking the long exact sequence of cohomology yields 
\[
    \dots \rightarrow H^1\big(X, \pi^*\Omega^1_{\PP^2} \otimes L\big) 
    \rightarrow H^1\big(X, \Omega^1_X \otimes L\big)
    \rightarrow H^1\big(X, L \otimes i_*\omega_{E/\PP^2}\big) 
    \rightarrow H^2\big(X, \pi^*\Omega^1_{\PP^2} \otimes L\big) 
    \rightarrow \dots
\]
Moreover, any line bundle $L$ on $X$ is of the form $\O_X(a\pi^*H + \sum_{i=1}^6 a_i E_i)$ where $H$ is the hyperplane divisor of $\PP^2$. Note that 
\[\O_X(1) \cong \O_X(3 \pi^* H - \sum_{i=1}^6 E_i).\]
Since $\O_{E_i}(E_i) \cong \O_{\PP^1}(-1)$, the condition $L(-2)$ ample implies that $a>6$ and $a_i + 2 < 0$ for all $1\le i \le 6$. By base change, we have 
\[R^1 \pi_* \O_X(\sum_{i=1}^6 a_i E_i) \cong \oplus_{i=1}^6 H^1(\PP^1, \O_{\PP^1}(-a_i)) = 0;\] so
\[
    H^1\big(X, \pi^*\Omega^1_{\PP^2} \otimes L\big) 
    \cong H^1\big(\PP^2, \Omega^1_{\PP^2} \otimes \pi_* L\big) 
    \cong H^1\big(\PP^2, \Omega^1_{\PP^2} \otimes \O_{\PP^2}(a)\big) = 0.
\]
On the other hand, since $\omega_{E/\PP^2} = \O_{\PP^3}(-2)|_E$, then 
\[H^1\big(X, L \otimes i_* \omega_{E/\PP^2}\big) \cong \oplus_{i=1}^6 H^1\big(\PP^1, \O_{\PP^1}(-a_i-2)\big) = 0.\] 
Thus we get the vanishing of $H^1\big(X, \Omega^1_X \otimes L\big)$. 

\textbf{Case 3}: If $d \ge 4$, we have the exact sequence 
\[
    \dots \rightarrow H^1\big(X, L(-d)\big)
    \rightarrow H^1\big(X, \Omega^1_{\PP^3}|_X \otimes L\big)
    \rightarrow H^1\big(X, \Omega^1_X \otimes L\big)
    \rightarrow H^2\big(X, L(-d)\big) \rightarrow \dots
\]
By the Nakano vanishing theorem, $H^1\big(X, L(-d)\big) = H^2\big(X, L(-d)\big) = 0$ if $L(4-2d)$ is ample, so it suffices to prove the vanishing of $H^1\big(X, \Omega^1_{\PP^3}|_X \otimes L\big)$.
    
Consider the exact sequence
\[
    0 \rightarrow \Omega^2_{\PP^3}|_X \otimes L \rightarrow \bigwedge^2 V \otimes L(-2) \rightarrow \Omega^1_{\PP^3}|_X \otimes L \rightarrow 0,
\]
where $V = H^1(X, \O_X(1))$. Its associated long exact sequence is 
\[
    \dots \rightarrow \bigwedge^2 V \otimes H^1\big(X, L(-2)\big) 
    \rightarrow H^1\big(X, \Omega^1_{\PP^3}|_X \otimes L\big) 
    \rightarrow H^2\big(X, \Omega^2_{\PP^3}|_X \otimes L\big)
    \rightarrow \bigwedge^2 V \otimes H^2\big(X, L(-2)\big)
    \rightarrow \dots
\]
By the Nakano vanishing theorem, the first and the last terms vanish if $L(2-d)$ is ample, which holds under the assumption because $d \ge 2$. Moreover, $H^2\big(X, \Omega^2_{\PP^3}|_X \otimes L\big)$ sits in the exact sequence 
\[
    \dots \rightarrow H^2\big(X, \Omega^1_X\otimes L(-d)\big) 
    \rightarrow H^2\big(X, \Omega^2_{\PP^3}|_X \otimes L\big)
    \rightarrow H^2\big(X, \Omega^2_X \otimes L\big) \rightarrow \dots
\]
Again, the Nakano vanishing theorem applies in this case and we get $H^2\big(X, \Omega^2_{\PP^3}|_X \otimes L\big) = 0$ which implies $ H^1\big(X, \Omega^1_{\PP^3}|_X \otimes L\big) = 0$; the theorem is proved. 
\end{proof}


\vspace{\bigskipamount}

\section{The vanishing theorem for Hodge ideals of $\QQ$-divisors with isolated singularities} \label{sec_iso4}

In this section, we prove Theorem \ref{vanishing_isolated_sing}. Notice that in the case $d=1$, i.e $X = \PP^{n-1}$, the condition in part $(3)$ is simply $l \ge m-n-1$ (this condition will appear in the proof), in which case the assertion is a result due to Musta\c t\u a and Popa, see \cite[Variant 12.5]{Hodge_ideals_Q1} and \cite[Variant 4.5]{BingYi}. Thus in this section, we only consider the case $d \ge 2$.

By Remark \ref{I_0}, the vanishing of $I_0(D)$ is simply a consequence of the Nadel vanishing theorem:
\[
     H^1\big(X, \omega_X(D) \otimes I_0(D) \otimes \O_X(l)\big) = 0
\]
if $l+\epsilon m > 0$ for any $0 < \epsilon \ll 1$, which is equivalent to $l\ge 0$.
    
\vspace{\medskipamount}
    
Moreover, if $I_k(D) = \O_X(Z-\cD) \cong \O_X(-a)$, then our vanishing result is simply a consequence of the Kodaira vanishing theorem 
\[
    H^i\big(X, \omega_X(kZ) \otimes I_k(D) \otimes \O_X(l)\big) \cong H^i\big(X, \omega_X(kZ) \otimes \O_X(l-a)\big) = 0
\]
for all $l$ if $1 \le i \neq n-1$ and for $(l-a)H + kZ$ ample if $i=n-1$. Hence, the theorem is trivial in this case. Thus we may assume that $I_k(D) \subsetneq \O_X(Z-\cD)$. Recall that the ideals $I_k(D) \otimes \O_X(\cD - Z)$ are cosupported on the singular locus of $D$, then $\dim(Z_k) = 0$, and hence all higher cohomologies of quasi-coherent sheaves on $Z_k$ vanish. This argument leads to the following lemma:

\begin{sublem} \label{use Nakano and isosing}
    Under the same assumptions in Theorem \ref{vanishing_isolated_sing}, we have
    \[
        H^i\big(X, \Omega_X^j \otimes I_k(D) \otimes \O_X(l)\big) = 0
    \]
    for all $i \ge 2$, $i+j \ge n$ and $l> a$.
\end{sublem}
\begin{proof}
    This is an immediate application of the Nakano vanishing theorem and the above argument to the exact sequence
    \[
        \dots \rightarrow H^{i-1}\big(Z_k, \Omega_X^j(Z-\cD) \otimes \O_X(l)|_{Z_k}\big)
        \rightarrow H^i\big(X, \Omega_X^j \otimes I_k(D) \otimes \O_X(l)\big) 
        \rightarrow H^i\big(X, \Omega_X^j(Z-\cD) \otimes \O_X(l)\big) \rightarrow \dots
    \]
\end{proof}

\begin{proof}[Proof of Theorem \ref{vanishing_isolated_sing}]
Parts $(1)$ and $(2)$: The term $H^i\big(X, \O_X(kZ) \otimes I_k(D) \otimes O_X(l)\big)$ sits in the following exact sequence
    \begin{flushleft}
        $\dots \rightarrow H^{i-1}\big(Z_k, \O_X((k+1)Z-\cD) \otimes O_X(l)|_{Z_k}\big)
        \rightarrow H^i\big(X, \O_X(kZ) \otimes I_k(D) \otimes O_X(l)\big) \rightarrow $    
    \end{flushleft}
    \begin{flushright}
        $\rightarrow H^i\big(X, \O_X((k+1)Z-\cD) \otimes O_X(l)\big) \rightarrow \dots$
    \end{flushright}
The first term vanishes since $i-1 > \dim(Z_k)$, while the last term 
\[
    H^i\big(X, \O_X((k+1)Z-\cD) \otimes O_X(l)\big) 
\]
vanishes for all $l$ if $0 < i \neq n-1$, and for $(l-a+n+1-d)H + kZ$ ample if $i = n-1$, by the Nakano vanishing theorem; the assertions $(1)$ and $(2)$ follow. 

\vspace{\medskipamount}

Part $(3)$: Let $\l = l+n+1-d$ and consider the vanishing of $H^1\big(X, \omega_X(kZ) \otimes I_k(D) \otimes \O_X(\tilde{l})\big)$. We will argue by induction on $k$, using the short exact sequence
\[
    0 \rightarrow \omega_X((k-1)Z) \otimes I_{k-1}(D) \otimes \O_X(\l)
    \rightarrow \omega_X(kZ) \otimes I_k(D) \otimes \O_X(\l)
    \rightarrow \M_k(D) \otimes \O_X(\l)
    \rightarrow 0. 
\]
Suppose the theorem holds for $I_{k-1}(D)$, then it suffices to show the vanishing of
\[
    H^1\big(X, \M_k(D) \otimes \O_X(\l)\big)
\]
for $\l \ge \max\{k(d-2) + m, k(d-2) + (n-2)(d-1) + a\}$. Consider the complex $C^\bullet = Sp_{k-n+1}(\M_{\bullet}(D)) \otimes \O_X(\l))[-k]$:
\[
    [\M_{0}(D) \otimes \bigwedge^{k} T_X \otimes \O_X(\l) \rightarrow \M_{1}(D) \otimes \bigwedge^{k-1} \otimes \O_X(\l)
    \rightarrow \dots 
    \rightarrow \M_k(D) \otimes \O_X(\l)] 
\]

\vspace{\medskipamount}

\noindent concentrated in degrees $0$ to $k$, with the associated spectral sequence of $C^\bullet$:
\[
    E_1^{p,q} = H^q(X, C^p) \Rightarrow \mathbb{H}^{p+q}(X, C^\bullet).
\] 
We are interested in the vanishing of $E_1^{k,1}$. Note that we have the vanishing of $E_1^{k+1,j} = E_1^{-1,j} = 0$ for all $j$ since $C^{k+1} = C^{-1} = 0$. Moreover, Theorem \ref{Saito_type_vanishing} implies $\mathbb{H}^j(X, C^\bullet) = 0$ for $j \ge k+1$ and $\l\ge m$, so it suffices to prove $E_1^{k,1} = E_\infty^{k,1}$, or more precisely, $E_r^{k,1} = E_{r+1}^{k,1}$ for all $r \ge 1$. 

\noindent Observe that $E_r^{k,1} = E_{r+1}^{k,1}$ for all $r > k$ since $C^\bullet$ has length $k+1$, and $E_r^{k+r,2-r} = 0$ because $C^{k+r} = 0$ for all $r \ge 1$. Thus it suffices to show that $E_1^{k-r,r} = 0$ for $1 \le r \le k$ and $\l \ge \max\{k(d-2) + m, k(d-2) + (n-2)(d-1) + a\}$. Now, the term $E_1^{k-r,r} = H^{r}\big(X, \M_{k-r} \otimes \bigwedge^r T_X \otimes \O_X(\l)\big)$ sits in the exact sequence 
\begin{flushleft}
    $\dots \rightarrow H^{r}\big(X, \Omega^{n-1-r}_X((k-r)Z)\otimes I_{k-r}(D) \otimes \O_X(\l)\big) 
    \rightarrow E_1^{k-r,r} \rightarrow$    
\end{flushleft}
\begin{flushright}
    $\rightarrow H^{1+r}\big(X, \Omega^{n-1-r}_X((k-r-1)Z) \otimes I_{k-r-1}(D) \otimes \O_X(\l)\big)
    \rightarrow \dots$,
\end{flushright}
in which the last term vanishes if $(\l-a)H+(k-r-1)Z$ is ample due to Lemma \ref{use Nakano and isosing} (note that the last term is $0$ if $r=k$). On the other hand, the first term sits in the exact sequence
\begin{flushleft}
    $\dots \rightarrow H^{r-1}\big(Z_{k-r}, \Omega^{n-1-r}_X((k-r)Z) \otimes \O_X(\l-a)|_{Z_{k-r}}\big) \rightarrow $
\end{flushleft}
\[
    \rightarrow H^r\big(X, \Omega^{n-1-r}_X((k-r)Z) \otimes \O_X(\l) \otimes I_{k-r}(D)\big) \rightarrow 
\]
\begin{flushright}
    $\rightarrow H^r\big(X, \Omega^{n-1-r}_X((k-r)Z) \otimes \O_X(\l-a)\big)
    \rightarrow \dots$
\end{flushright}
By Theorem \ref{Bott vanishing hypersurfaces}, the last term vanishes if $(\l-a+n-(n-r)d)H+(k-r)Z$ is nef which holds under our assumption. If $r\ge 2$, the first term also vanishes since $\dim(Z_{k-r}) = 0$, and we are done. Thus we are left with the case $r=1$. By comparing to the corresponding exact sequence with $\Omega^{n-2}_{X}$ replaced by $\Omega^{n-2}_{\PP^n}|_X$, we get the commutative diagram: 
\[\begin{tikzcd}
	\begin{tabular}{c}
	     $H^0\big(X, \Omega_X^{n-2}((k-1)Z)$ \\ $\otimes$ \\ $\O_X(\l-a)\big)$  
	\end{tabular}
	& 
	\begin{tabular}{c}
	     $H^0\big(Z_0, \Omega_X^{n-2}((k-1)Z)$ \\ $\otimes$ \\ $\O_X(\l-a)|_{Z_0}\big)$
	\end{tabular}
	& 
	\begin{tabular}{c}
	     $H^1\big(X, \Omega^{n-2}_X((k-1)Z)$ \\ $\otimes$ \\ $\O_X(\l) \otimes I_{k-1}(D)\big)$
	\end{tabular} 
	\\
	\begin{tabular}{c}
	     $H^0\big(X, \Omega_{\PP^n}^{n-2}((k-1)Z)|_X$ \\ $\otimes$ \\ $\O_X(\l-a) \big)$  
	\end{tabular}
	& 
	\begin{tabular}{c}
	     $H^0\big(Z_0, \Omega_{\PP^n}^{n-2}((k-1)Z)|_X$ \\ $\otimes$ \\ $\O_X(\l-a)|_{Z_0}\big)$
	\end{tabular}
	& 
	\begin{tabular}{c}
	     $H^1\big(X, \Omega^{n-2}_{\PP^n}((k-1)Z)|_X$ \\ $\otimes$ \\ $\O_X(\l) \otimes I_{k-1}(D)\big)$
	\end{tabular}
	\arrow[from=1-1, to=1-2]
	\arrow[from=2-1, to=2-2]
	\arrow[from=2-2, to=2-3]
	\arrow[two heads, from=1-2, to=1-3]
	\arrow[from=2-1, to=1-1]
	\arrow[two heads, from=2-2, to=1-2]
\end{tikzcd}\]
where the surjectivity of the second vertical arrow is due to the vanishing of $H^1\big(Z_0, \Omega_X^{n-3}((k-1)Z) \otimes \O_X(\l-a-d)|_{Z_0}\big)$. Hence the vanishing of $H^1\big(X, \Omega^{n-2}_{\PP^n}((k-1)Z)|_X \otimes \O_X(\l) \otimes I_{k-1}(D)\big)$ implies the vanishing of $H^1\big(X, \Omega^{n-2}_X((k-1)Z) \otimes \O_X(l) \otimes I_{k-1}(D)\big)$. 
By the spectral sequence associated to the Koszul-type resolution 
\[
    \bigoplus \O_X(-n-1) \rightarrow \bigoplus \O_X(-n) \rightarrow \bigoplus \O_X(-n+1)
\]
of $\Omega_{\PP^n}^{n-2}|_X$, this is a consequence of the vanishing of 
\[
    H^j\big(X, \O_X(\l-n+2-j) \otimes \O_X((k-1)Z) \otimes I_{k-1}(D)\big)
\]
for all $1 \le j \le 3$ (the important one is $j = 1$). All of these conditions are satisfied by the assumption; the proof is complete. 
\end{proof}


\vspace{\bigskipamount}

\section{The vanishing theorem for Hodge ideals on hypersurfaces in $\PP^3$} \label{sec_iso3}

In this case, the divisor $D$ has dimension $1$ and hence automatically has at worst isolated singularities.

\begin{proof}[Proof of Theorem \ref{vanishing_surface}]
$(1)$ For $i\ge 2$, the result follows from the Nakano vanishing theorem as in Lemma \ref{use Nakano and isosing}.

$(2)$ For $i=1$, we will again argue by induction on $k$. We have seen the base case $k=0$ in Remark \ref{I_0}. Suppose the theorem holds up to $k-1$, then it suffices to prove the vanishing for $\M_k(D) \otimes L$ under our assumptions.

Consider the complex $C^\bullet = (Sp_{k-2}(\M_{\bullet}(D)) \otimes L)[-2]$:
\[
    [\M_{k-2} \otimes \bigwedge^2 T_X \otimes L  \rightarrow \M_{k-1}(D) \otimes T_X \otimes L \rightarrow \M_k(D) \otimes L]  
\]
concentrated in degrees $0$ to $2$. The associated spectral sequence of $C^\bullet$ is 
\[
    E_1^{p,q} = H^q(X, C^p) \Rightarrow \mathbb{H}^{p+q}(X, C^\bullet).
\] 
Again, we are interested in the vanishing of $E_1^{2,1}$, and by similar argument, we only need to show that $E_1^{1,1} = E_1^{0,2} = 0$ under the given hypothesis. 

First, we prove $E_1^{1,1} = 0$. Indeed, it sits in the exact sequence
\[
    \dots \rightarrow H^{1}\big(X, \Omega^{1}_X \otimes L((k-1)Z) \otimes I_{k-1}(D)\big) 
    \rightarrow E_1^{1,1} 
    \rightarrow H^{2}\big(X, \Omega^{1}_X \otimes L((k-2)Z) \otimes I_{k-2}(D)\big)
    \rightarrow \dots,
\]
in which the last term vanishes if $L((k-1)Z-\cD)$ is ample due to the Nakano vanishing theorem (note that the last term is $0$ if $k=1$). The first term sits in the exact sequence 
\[\begin{tikzcd}[column sep = 3mm]
	\dots 
	& \begin{tabular}{c}
	     $H^1\big(X, \Omega^{1}_{\PP^3}|_X \otimes I_{k-1}(D) $ \\ $\otimes L((k-1)Z)\big)$
	\end{tabular} 
	& \begin{tabular}{c}
	     $H^{1}\big(X, \Omega^{1}_X \otimes I_{k-1}(D)$ \\ $\otimes L((k-1)Z) \big)$ 
	\end{tabular} 
	& \begin{tabular}{c}
	    $H^2\big(X, \O_X(-d) \otimes I_{k-1}(D)$ \\ $\otimes L((k-1)Z) \big)$    
	\end{tabular} 
	& \dots
	\arrow[from=1-2, to=1-3]
	\arrow[from=1-1, to=1-2]
	\arrow[from=1-3, to=1-4]
	\arrow[from=1-4, to=1-5]
\end{tikzcd}\]
By part $(1)$, the last term vanishes if $L(kZ-\cD)(4-2d)$ ample. To obtain the vanishing of the first term, by using the Koszul-type resolution of $\Omega_{\PP^3}^1|_X$, it is enough to have
\[
    H^j\big(X, L((k-1)Z)(-1-j) \otimes I_{k-1}(D)\big) = 0
\]
for all $1 \le j \le 3$. The vanishing for $j\in \{2,3\}$ follows from part $(1)$ and the case $j=1$ follows from the induction hypothesis.

Next, we finish the proof by showing the vanishing of $E_1^{0,2}$. Using the exact sequence
\[
    \dots \rightarrow H^{2}\big(X, L((k-2)Z) \otimes I_{k-2}(D)\big) 
    \rightarrow E_1^{0,2} 
    \rightarrow H^{3}\big(X, L((k-3)Z) \otimes I_{k-3}(D)\big) = 0,
\]
we are left with the vanishing of the first term. Again, by the Nakano vanishing theorem, 
\[
    H^2\big(X, L((k-2)Z) \otimes I_{k-2}(D)\big) = 0
\]
if $L((k-1)Z-\cD)(4-d)$ is ample; the assertion $(2)$ follows. 
\end{proof}

\vspace{\bigskipamount}


\section{Applications}

It is well known that the set of singular points of a nodal irreducible plane curve of degree $m$ imposes independent conditions on the linear system of plane curves of degree $m-3$. In the paper \cite{Severi}, Severi proved a similar result in dimension $2$, stating that the set of singular points of nodal hypersurfaces of degree $m$ in $\PP^3$ imposes independent conditions on hypersurfaces of degree at least $2m-5$. Later, Park and Woo generalized Severi's result to all singular points, and gave similar bounds for isolated singular points on hypersurfaces in $\PP^n$, see \cite{ParkWoo}. By using the Hodge ideals, Musta\c t\u a and Popa gave a different bound for this problem. In this section, we exploit Theorem \ref{vanishing_isolated_sing} to obtain a new result for isolated singular points on the complete intersection of two hypersurfaces in $\PP^n$.

\begin{proof} [Proof of Theorem \ref{bound of deg}]
The triviality of the Hodge ideals $I_k(D)$ is equivalent to the scheme $Z_k$ being empty (due to \cite[Theorem A]{Hodge_ideals} for the reduced divisor $D$), hence there is nothing to prove. In the case $Z_k$ is non-empty, denote by $A$ the line bundle 
\[
    \omega_X \otimes \O_X((k+1)D) \otimes \O_X(l) = \O_X((k+1)m + l + d - n -1).
\]
Consider the short exact sequence 
\[
    0 \rightarrow A \otimes I_k(D) 
    \rightarrow A 
    \rightarrow A \otimes \O_{Z_k} 
    \rightarrow 0.
\]

For $(1)$, Theorem \ref{vanishing_isolated_sing} implies that 
\[
    H^1(X, A \otimes I_k(D)) = 0.
\]
Using the associated long exact sequence
\[
    0 \rightarrow H^0(X, A \otimes I_k(D))
    \rightarrow H^0(X, A) 
    \rightarrow H^0(Z_k, A \otimes \O_{Z_k})
    \rightarrow H^1(X, A \otimes I_k(D)) \rightarrow \dots,
\]
we see that $\text{length}(Z_k)$ is bounded by the dimension of $H^0(X, A)$ which is exactly 
\[\binom{(k+1)m +l +d - 1}{n} - \binom{(k+1)m +l - 1}{n}.\]

\vspace{\medskipamount}

For $(2)$, we know by \cite[Corollary 21.3]{Hodge_ideals} that 
\[
    I_k(D) \subseteq I_{S_t},
\]
where $k = \lceil \frac{n}{t} - 1\rceil$. Since the restriction map
\[
    H^0\big(\PP^n, \O_{\PP^n}((k+1)m + l + d - n -1)\big) \twoheadrightarrow H^0(X, A)
\]
is surjective, the result follows immediately from part $(1)$ which says that there is a surjection
\[
    H^0(X, A) \twoheadrightarrow \O_{Z_k} \rightarrow 0.
\]
\end{proof}

We can extend this analysis to the study of $(j-1)$-jets along $S_t$, just as it was done in \cite[Corollary 27.3]{Hodge_ideals}. Recall that for a $0$-dimensional subset $S$ of $\PP^n$, the space of hypersurfaces of degree $a$ is said to separate $(j-1)$-jets along $S$ if the restriction map 
\[
    H^0(\PP^n, \O_{\PP^n}(a)) \rightarrow \bigoplus_{x \in S} \O_{\PP^n}/\mathfrak{m}_x^{j}
\]
is surjective. In particular, $S$ imposes independent conditions on hypersurfaces of degree $a$ if they separate $0$-jets along $S$. For $t\ge 3$, we put
\[
    k_{t,j} \coloneqq 
    \begin{cases}
        \lceil \frac{n - 1 - t + j}{t} \rceil & j \le t-1 \\
        \lceil \frac{n - 1 - t + j}{t-2} \rceil & j \ge t
    \end{cases}.
\]
We prove the following.

\begin{subprop} \label{bound of deg for jet}
    Let $D$ be a complete intersection of a smooth irreducible hypersurface $X$ of degree $d \ge 2$ and another hypersurface (not necessarily smooth) of degree $m$ in $\PP^n$, with $n \ge 4$. Assume furthermore that $D$ has at worst isolated singularities. Let $S_t$ denote the set of singular points on $D$ of multiplicity at least $t \ge 3$, and set $l = \max \{k_{t,j}(d-2), k_{t,j}(d-2) + (n-2)(d-1) - m\}$. Then the space of hypersurfaces of degree at least 
    \[
        (k_{t,j}+1)m + l + d - n - 1 
    \]
    separates $(j-1)$-jets along $S_t$, for each $j\ge 1$.  
\end{subprop}

\begin{proof}
    The proof is analogous to the proof of Theorem \ref{bound of deg}, except that we use both \cite[Theorem E]{Hodge_ideals} and \cite[Corollary 19.4]{Hodge_ideals} to get 
\[
    I_{k_{t,j}}(D) \subseteq \mathfrak{m}_x^j
\]
for every $x \in S_t$.
\end{proof}

\begin{remark}
    \begin{enumerate}
        \item In the case $d = 1$ i.e $D$ is a hypersurface of $\PP^{n-1}$, Theorem \ref{bound of deg} specializes to \cite[Theorem G]{Hodge_ideals} for $z_k = 0$ and \cite[Corollary H]{Hodge_ideals}, while Proposition \ref{bound of deg for jet} specializes to \cite[Corollary 27.2]{Hodge_ideals}.
        
        \item The conditions in Theorem \ref{bound of deg} and Proposition \ref{bound of deg for jet} are not symmetric in $d$ and $m$. In fact, the bound tends to increase faster as $m$ increases, compared to $d$. This is a limitation of our method; thus in the case when $D$ is the intersection of two smooth hypersurfaces, it is much better to consider $D$ as a divisor in the smooth hypersurface with higher degree. In order to obtain symmetric bounds, a new method needs to be employed.
    \end{enumerate}
\end{remark}

\begin{example}
\hfill
    \begin{enumerate}
        \item Let $d=m=2$ and $n=4$, so that $D$ is a complete intersection of two quadrics in $\PP^4$. Theorem \ref{bound of deg} says that: 
        \begin{itemize}
            \item $S_2$ and $S_3$ impose independent conditions on hypersurfaces of degree at least $m-1$.
            \item If $t = 4$, Theorem \ref{bound of deg} shows that
            \[
                0 = H^0(X, \O_X(-1)) \twoheadrightarrow \O_{S_t}
            \]
            which means $D$ has no isolated singularities of multiplicity at least $4$.  
        \end{itemize} 
    
        \item In general, if $d=2, m \le n-2$ and $t=n$, we obtain that any reduced effective Cartier divisor on a smooth quadric hypersurface of $\PP^n$ has no isolated singularities of multiplicity at least $n$. This result can also be proven by using the Nadel vanishing theorem. 
    \end{enumerate}
\end{example}


\vspace{\bigskipamount}

\section{The vanishing theorem for divisors with arbitrary singularities} \label{sec_arbsing}

The aim of this section is to give a vanishing theorem for Hodge ideals of a divisor in a hypersurface of $\PP^n$, with arbitrary singularities. While the situation is much more complicated technically than in the isolated singularity case, we can at least provide an algorithmic approach. For simplicity, we consider the case when $D$ is a reduced effective divisor. The case when $D$ is an effective $\QQ$-divisor can be approached in the same manner but is more complicated. We will use the same method as in Section \ref{sec_iso4}. First, we proceed by induction on $k$. Given the vanishings for $I_0(D),\dots, I_{k-1}(D)$ i.e the vanishings of $H^i\big(X, \omega_X(kD) \otimes I_j(D) \otimes \O_X(l)\big)$ for $0\le j\le k-1$ and appropriate $l$ obtained from the induction process, it is equivalent to study the vanishing of $\text{gr}^F_{k}\O_X(*D)$. Next, by using Theorem \ref{Saito_type_vanishing} and the associated spectral sequence of the graded de Rham complex 
\[(\text{gr}^F_{k+1-n}\text{DR}(\O_X(*D)) \otimes \O_X(l))[-k],\]
it suffices to prove the vanishing of 
\[
    E^{k-r,i+r-1}_1 = H^{i+r-1}\big(X, \Omega_X^{n-1-r} \otimes \O_X(l) \otimes \operatorname{gr}^F_{k-r}(\O_X(*D))\big)
\]
for all $1 \le r \le k$. However, since we are not imposing any restriction on the singularities of the divisor $D$, we only have
\[
    \dim(Z_k) \le \dim(D)-1 = n-3,
\]
and this will impose more conditions in the inductive steps. 

For fixed $n,d,m$, denote by $A_{i,k,l}$ the statement 
\[H^i\big(X, \omega_X \otimes \O_X((k+1)m+l) \otimes I_k(D)\big) = 0.\]
By using a similar strategy as in the proof of Theorem \ref{vanishing_isolated_sing}, we will prove the following:

\begin{subprop}
    In the above setting, the following hold: 
    \begin{enumerate} 
    \item For $i\ge 2$,
        \[A_{i,k-1,l} \quad \text{and} \quad A_{i-1,k-r,rd} \quad \text{ for } \quad 1\le r\le k\] and 
        \[A_{j+t-1,k-r,l+(i+r-j)d-i+j-t} \quad \text{ for } \quad 1\le r\le k, i+1\le j \le i+r, 0\le t\le j-i\]
        together imply $A_{i,k,l}$.
    \item For $i=1$,
        \[A_{1,k-1,l+rd} \quad \text{ for } \quad 0\le r\le k-1,\]
        \[A_{j+t, k-r-1, l+(r-j+1)d+j-t-1} \quad \text{ for } \quad 1\le r\le k-1, i+1\le j \le i+r, 0\le t \le j-i,\]
        \[A_{j+t, k-r, l+(r-j-1)d+j-t+1} \quad \text{ for } \quad 1\le r\le k, r\le j\le n-3, 0\le t\le j+1,\]
        and 
        \[l\ge (n-r)d-n-(k-r+1)m\]
        together imply $A_{i,k,l}$.
\end{enumerate}
\end{subprop}

\noindent Notice that in the above proposition, the conditions that guarantee $A_{i,k,l}$ contain the ones that guarantee both $A_{i+1,k,l}$ and $A_{i,k-1,l}$. Roughly speaking, we need "more conditions" as $i$ gets smaller and $k$ gets larger. 

\begin{proof}
    In the base case $I_0(D)$, we see as in Remark \ref{I_0} that
\[
    H^i\big(X, \omega_X(D) \otimes \O_X(l) \otimes I_0(D)\big) = 0
\]
for all $i \ge 1$ and $l \ge 0$ by the Nadel vanishing theorem.

\vspace{\medskipamount}

For the induction steps, using the exact sequence 
\[\begin{tikzcd}[column sep = 3mm]
    \tag{7.1} \label{7.1}
	& \begin{tabular}{c}
	     $H^{i+r-1}\big(X, \Omega^{n-1-r}_X \otimes I_{k-r}(D)$ \\ 
	     $\otimes \O_X(l+(k-r+1)m)\big)$
	\end{tabular} 
	& {E_1^{k-r,i+r-1}}
	& \begin{tabular}{c}
         $H^{i+r}\big(X, \Omega^{n-1-r}_X \otimes I_{k-r-1}(D)$  \\
         $\otimes \O_X(l+(k-r)m)\big)$
    \end{tabular} 
	\arrow[from=1-2, to=1-3]
	\arrow[from=1-3, to=1-4]
\end{tikzcd},\]
we see that the vanishings of both extremal terms induce the vanishing of $E^{k-r, i+r-1}_1$.

\vspace{\medskipamount}

$\bullet$ Consider the vanishing of the last term of (\ref{7.1}) 
\[
    H^{i+r}\big(X, \Omega^{n-1-r}_X \otimes \O_X(l+(k-r)m) \otimes I_{k-r-1}(D)\big).
\]
This term is $0$ if $r=k$. Using the associated long exact sequence on cohomology of the exact sequences
\[
    0 \rightarrow \Omega^{p-1}_X \otimes \O_X(-d) \rightarrow \Omega^{p}_{\PP^n}|_X \rightarrow \Omega^p_X \rightarrow 0
\]
for $n-r \le p \le n-1$, we see that we only need to show  

\[
    H^{i}\big(X, \omega_X((k-r)D) \otimes \O_X(l+rd) \otimes I_{k-r-1}(D)\big) = 0
\]
and 
\[
    H^{j}\big(X, \Omega^{n+i-j}_{\PP^n}|_X \otimes \O_X(l+(k-r)m+(i+r+1-j)d) \otimes I_{k-r-1}(D)\big) = 0
\]
for all $i+1 \le j \le i+r$. Using the spectral sequence of the Koszul-type resolution of $\Omega_{\PP^n}^{n+i-j}|_X$
\[
    0 \rightarrow \bigoplus \O_X(-n-1)
    \rightarrow \dots
    \rightarrow \bigoplus \O_X(-n-i+j-1)
    \rightarrow \Omega^{n+i-j}_{\PP^n}|_X
    \rightarrow 0,
\]
we see that these conditions are induced by the vanishings for lower index Hodge ideals.
\[
    H^{j+t}\big(X, \omega_X((k-r)D) \otimes \O_X(l + (i+r-j)d-i+j-t) \otimes I_{k-r-1}(D)\big) = 0
\]
for all $i+1 \le j \le i+r$ and $0 \le t \le j-i$. 

\vspace{\medskipamount}

$\bullet$ Consider the vanishing of the first term of (\ref{7.1})
\[
     H^{i+r-1}\big(X, \Omega^{n-1-r}_X \otimes \O_X(l+(k-r+1)m) \otimes I_{k-r}(D)\big).
\]
There are two cases:

\underline{Case 1}: If $i\ge 2$, arguing as above, we need
\[
    H^{i-1}\big(X, \omega_X((k-r+1)D) \otimes \O_X(l+rd) \otimes I_{k-r}(D)\big) = 0
\]
and 
\[
    H^{j}\big(X, \Omega^{n+i-1-j}_{\PP^n}|_X \otimes \O_X(l+(k-r+1)m+(i+r-j)d) \otimes I_{k-r}(D)\big) = 0
\]
for all $i \le j \le i+r-1$ which are implied by the following vanishings:
\[
    H^{j+t}\big(X, \omega_X((k-r+1)D) \otimes \O_X(l + (i+r-j-1)d -i+1+j-t) \otimes I_{k-r}(D)\big) = 0
\]
for all $i \le j \le i+r-1$ and $0 \le t \le j+1-i$.

\underline{Case 2}: If $i=1$, we require the vanishings
\[
    H^{j}\big(X, \Omega^{n-1-j}_{\PP^n}|_X \otimes \O_X(l+(k-r+1)m -(j-r)d) \otimes I_{k-r}(D)\big) = 0
\]
for $r \le j \le n-3$ and
\[
    H^{n-2}\big(X, \Omega_X^1 \otimes \O_X(l+(k-r+1)m-(n-2-r)d) \otimes I_{k-r}(D)\big) = 0.
\]
Since $\dim(Z_{k-r}) \le n-3$, by using the same argument at the end of the proof of Theorem \ref{vanishing_isolated_sing}, the vanishing of 
\[
    H^{n-2}\big(X, \Omega^1_X \otimes \O_X(l+(k-r+1)m-(n-2-r)d) \otimes I_{k-r}(D)\big)
\]
is implied by the vanishings of
\[
    H^{n-2}\big(X, \Omega^1_{\PP^n}|_X \otimes \O_X(l+(k-r+1)m-(n-2-r)d) \otimes I_{k-r}(D)\big)
\] and 
\[
    H^{n-2}\big(X, \Omega^1_X \otimes \O_X(l+(k-r+1)m-(n-2-r)d)\big),
\]
which follow from Theorem \ref{Bott vanishing hypersurfaces}.
Thus, in this case, we require the following set of conditions:
\[
    l \ge (n-r)d - n - (k-r+1)m
\]
and
\[
    H^{j+t}\big(X, \omega_X((k-r+1)D) \otimes \O_X(l - (j-r+1)d  + j + 1 - t)) \otimes I_{k-r}(D)\big) = 0
\]
for all $r \le j \le n-3$ and $0 \le t \le j + 1$.
\end{proof}
 

\vspace{\medskipamount}

Although the algorithm looks complicated, it turns out that many things simplify when implemented in special cases. For example, we can work out $A_{i,1,l}$ as follows:
\begin{subprop}
    Let $X$ be a smooth hypersurface of degree $d\ge 2$ in $\PP^n$ with $n \ge 4$, and $D$ a reduced effective divisor. Let $\O_X(D) = \O_X(m)$ for some positive integer $m$. Then: 
    \begin{enumerate}
        \item For $i\ge 2$ and $l \ge 0$
        \[
            H^i(X, \omega_X(2D) \otimes I_1(D) \otimes \O_X(D)) = 0.
        \]
        \item For $l \ge \max\{(n-2)d, (n-1)d - n -m\}$, then  
        \[
            H^1(X, \omega_X(2D) \otimes I_1(D) \otimes \O_X(D)) = 0. 
        \] 
    \end{enumerate}
\end{subprop}

\vspace{\bigskipamount}

\printbibliography

\Addresses
\end{document}